\documentclass{amsart}
\usepackage{amssymb}
\usepackage{amsfonts}

\newtheorem{theorem}{Theorem}
\theoremstyle{plain}

\newtheorem{definition}{Definition}

\newtheorem{remark}{Remark}

\numberwithin{equation}{section}
\input{tcilatex}

\newcommand{\R}{{\mathbb{R}}}
\newcommand{\C}{{\mathbb{C}}}
\newcommand{\N}{{\mathbb{N}}}
\newcommand{\e}{\mathsf{e}}

\DeclareMathOperator{\pow}{\mathrm{pow}}
\DeclareMathOperator{\add}{\mathrm{add}}

\begin{document}
\title[Cauchy Pairs]{Cauchy Pairs\\
}
\author{Martin Himmel}
\curraddr{Technical University Mountain Academy Freiberg,
% Faculty of Mathematics and Computer Science, Intitute of Applied Analysis, 
% Nonnengasse 22,
Tschaikowskistraße 67 
 09596 Freiberg,
 Germany}
\email{martin.himmel@gmail.com}

\begin{abstract}
	The notion of pairable functions is introduced and some of its properties are developed.
	In this connection the famous Euler identity is interpreted as a property of certain pairable functions and finite cyclic groups.
\end{abstract}

\maketitle

\QTP{Body Math}
$\bigskip $

\section{Introduction}
In this note we start with a prelude on Euler´s reflection formula linking the Gamma function to trigonometric functions.
This allows us to reformulate Euler´s identity in terms of the Gamma function evaluated at certain affine functions summing to one.
We continue with a generalization of the classical Gamma function based on the original version due to Euler.
But the main topic of this paper is the notion of pairability of functions,
which in some sense is the the idea to translate or scale a given function under the condition that a certain functional equation is satisfied.

\subsection{Reflection formula for cosine  function and a reformulation of Euler's identity }
The well-known reflection formula
\begin{equation*}
	\Gamma{\left(z\right)} \Gamma{\left(1-z\right)}=\frac{\pi}{\sin\left(\pi z\right)}, \qquad z \in \C \setminus -\N_0 
\end{equation*}
linking the Euler Gamma function and the sine function has a beautiful analogon for the cosine function.
Substituting here $z$ by $\frac{1}{2}-z$, we get
\begin{equation*}
	\Gamma{\left(\frac{1}{2}-z\right)} \Gamma{\left(\frac{1}{2}+z\right)}=\frac{\pi}{\cos\left(\pi z\right)}.
\end{equation*}
Rearranging these equations for sine and cosine, respectively, gives us
\begin{equation*}
\sin{z}	=\frac{\pi}{\Gamma{\left(\frac{z}{\pi}\right)} \Gamma{\left(1-\frac{z}{\pi}\right) }},
\end{equation*}
and
\begin{equation*}
	\cos{z} = \frac{\pi}{\Gamma{\left(\frac{1}{2}-\frac{z}{\pi}\right)} \Gamma{\left(\frac{1}{2}+\frac{z}{\pi}\right)}},
\end{equation*}
and thus
\begin{equation*}
	\tan{z} = \frac{\Gamma{\left(\frac{z}{\pi}\right)} \Gamma{\left(1-\frac{z}{\pi}\right) }}{\Gamma{\left(\frac{1}{2}-\frac{z}{\pi}\right)} \Gamma{\left(\frac{1}{2}+\frac{z}{\pi}\right)}}
\end{equation*}

Hence, we may express the Euler identity $e^{iz}=\cos{z}+i \sin{z}$ without explicitly using trigonometric functions by 
\begin{equation*}
	e^{iz} = \frac{\pi}{\Gamma{\left(\frac{1}{2}-\frac{z}{\pi}\right)} \Gamma{\left(\frac{1}{2}+\frac{z}{\pi}\right)}} 
	+ i \frac{\pi}{	\Gamma{\left(\frac{z}{\pi}\right)} \Gamma{\left(1-\frac{z}{\pi}\right) }}.
\end{equation*}
Substituting now $z$ by $-z$ gives us
\begin{equation*}
	e^{-iz} = \frac{\pi}{\Gamma{\left(\frac{1}{2}+\frac{z}{\pi}\right)} \Gamma{\left(\frac{1}{2}-\frac{z}{\pi}\right)}} 
	+ i \frac{\pi}{	\Gamma{\left(\frac{-z}{\pi}\right)} \Gamma{\left(1+\frac{z}{\pi}\right) }}.
\end{equation*}
Multiplying these two equations with each other ($e^{iz}  e^{-iz} =1=\ldots $), we obtain
\begin{equation*}
	1= 
	\left( \frac{\pi}{\Gamma{\left(\frac{1}{2}-\frac{z}{\pi}\right)} \Gamma{\left(\frac{1}{2}+\frac{z}{\pi}\right)}} 
	+ i \frac{\pi}{	\Gamma{\left(\frac{z}{\pi}\right)} \Gamma{\left(1-\frac{z}{\pi}\right) }}\right)
	\left( \frac{\pi}{\Gamma{\left(\frac{1}{2}+\frac{z}{\pi}\right)} \Gamma{\left(\frac{1}{2}-\frac{z}{\pi}\right)}} 
	+ i \frac{\pi}{	\Gamma{\left(\frac{-z}{\pi}\right)} \Gamma{\left(1+\frac{z}{\pi}\right) }} \right)	.
\end{equation*}
Sorting now by real and immaginary part gives us the following two equations
\begin{equation*}
	1= 
	 \frac{\pi}{\Gamma{\left(\frac{1}{2}-\frac{z}{\pi}\right)} \Gamma{\left(\frac{1}{2}+\frac{z}{\pi}\right)}} 
	 \cdot
	 \frac{\pi}{\Gamma{\left(\frac{1}{2}+\frac{z}{\pi}\right)} \Gamma{\left(\frac{1}{2}-\frac{z}{\pi}\right)}}
	- \frac{\pi}{	\Gamma{\left(\frac{z}{\pi}\right)} \Gamma{\left(1-\frac{z}{\pi}\right) }}
	  \cdot \frac{\pi}{	\Gamma{\left(\frac{-z}{\pi}\right)} \Gamma{\left(1+\frac{z}{\pi}\right) }}, 
\end{equation*}
and
\begin{equation*}
	0= \frac{\pi}{\Gamma{\left(\frac{1}{2}-\frac{z}{\pi}\right)} \Gamma{\left(\frac{1}{2}+\frac{z}{\pi}\right)}} 
	\cdot \frac{\pi}{	\Gamma{\left(\frac{-z}{\pi}\right)} \Gamma{\left(1+\frac{z}{\pi}\right) }} +
	\frac{\pi}{	\Gamma{\left(\frac{z}{\pi}\right)} \Gamma{\left(1-\frac{z}{\pi}\right) }} 
	\cdot
	\frac{\pi}{\Gamma{\left(\frac{1}{2}+\frac{z}{\pi}\right)} \Gamma{\left(\frac{1}{2}-\frac{z}{\pi}\right)}}
\end{equation*}
On the other hand, the trigonometric version of the well-known Pythagorean theorem gives us
\begin{equation*}
	\frac{1}{\left( \Gamma \left( \frac{z}{\pi}\right) \Gamma \left(1-\frac{z}{\pi}\right)\right)^2}+\frac{1}{\left(\Gamma \left(\frac{1}{2} -\frac{z}{\pi}\right) \Gamma\left(\frac{1}{2}+\frac{z}{\pi}\right) \right)^2}=\frac{1}{\pi^2}.
\end{equation*}

\section{Euler Gamma function}
The Euler Gamma function $\Gamma: (0, + \infty) \to (0, + \infty)$ is usually defined by
\begin{equation*}
	\Gamma(x)=\int_0^\infty {e^{-t} t^{x-1}\, dt}, \qquad x \in (0, +\infty).
	%\label{eq:}
\end{equation*} 
An equivalent formulation,
\begin{equation*}
	\Gamma(x)=\int_0^1 \left({-\log{t}}\right)^{x-1}\, dt, \qquad x \in (0, +\infty),
	%\label{eq:GammaLog}
\end{equation*} 
shows us a connection to the classical logarithm function,  a regular solution to the Cauchy functional equation of logarithmic type.
This motivates us to introduce the following
\begin{definition}

Let $I \subset \R$ be an interval and $\varphi: I \to (0, +\infty)$ an integrable function.
We call
\begin{equation}
	\Gamma_{\varphi} (x)=\int_0^1 \left({\varphi{(t)}}\right)^{x-1}\, dt, \qquad x \in I,
	\label{eq:GammaPHI}
\end{equation} 
the Gamma function of generator $\varphi$.
\end{definition}
In this light, the Euler Gamma function equals the Gamma function of generator negative logarithm\footnote{For any $a>0$, the real logarithmic function $x \mapsto \log_a{x}$ is negative for $x\in (0, a)$  and positive for $x\in (a, +\infty)$. Thus, the Gamma integral is well-defined.},
i.e. $\Gamma_{-\log}=\Gamma$.
\subsection{Gamma function of regular Cauchy generator}
What kind of functions are obtained when
the generator $\varphi$  in \eqref{eq:GammaPHI} is a regular solution to one of the other three Cauchy equations
is discussed now briefly.

When $\varphi$ is a regular solution the Cauchy function equation of exponential type, i.e., if there is $a>0$ such that $\varphi(t)=a^t$ for all $t \in \R$,
we obtain 
\begin{equation*}
	\Gamma_{\exp} (x)=\int_0^1 \left({a^{t}}\right)^{x-1}\, dt, \qquad x\in \R,
	%\label{eq:GammaLog}
\end{equation*} 
as the Gamma function of exponential generator.
Since the integration here can be performed explicity,
we have, for all $x \in \R$, $x \neq 1$,
\begin{eqnarray*}
	\Gamma_{\exp} (x)&=& \left. \frac{{\left(a^t  \right)}^{x-1}}{(x-1)\log{a}} \right|_0^1\\
	&=&
	\frac{a^{x-1}-1}{(x-1) \log{a}},
	%\label{eq:GammaLog}
\end{eqnarray*} 
and $\Gamma_{\exp} (1)=1$.

If $\varphi$ is a regular solution the Cauchy functional equation of additive type, which means that there is $c \in \R$ such that $\varphi(t)=ct$ for all $t \in \R$,
we obtain, for all $ x\in \R$, $x \neq 0$,
\begin{eqnarray*}
	\Gamma_{\add} (x) &=& \int_0^1 \left({c{t}}\right)^{x-1}\, dt, \\
	&=& 
	\left.  \frac{(ct)^x}{cx} \right|_0^1\\
	&=&
	\frac{c^{x-1}}{x},
	%\label{eq:GammaLog}
\end{eqnarray*} 
as the Gamma function of additive generator.
For obvious reasons $c$ must be positive here.

If $\varphi$ is a regular solution to the Cauchy functional equation of logarithmic type, i.e.,
 there is $c \in \R$ such that $\varphi(t)=c\log{t}$ for all $t \in (0, +\infty)$,
we obtain, for all $ x\in \R \setminus (- \N)$,
\begin{eqnarray*}
	\Gamma_{\log} (x) &=& \int_0^1 \left({c\log{t}}\right)^{x-1}\, dt, \\
	&=& 
	(-c)^{x-1} \Gamma(x),
	%\label{eq:GammaLog}
\end{eqnarray*} 
as the Gamma function of logarithmic generator, and $c$ is restricted to be negative. 

If $\varphi$ is a regular solution to the Cauchy functional equation of multiplicative type, 
which means that there is $p \in \R$ such that $\varphi(t)={t}^p$ for all $t \in (0, +\infty)$,
we obtain, for all $ x\in \R$, 
\begin{eqnarray*}
	\Gamma_{\pow} (x) &=& \int_0^1 \left({{t}^p}\right)^{x-1}\, dt, \\
	&=& 
	\left. \frac{t^{p (x-1)+1}}{p(x-1)+1} \right|_0^1\\
	&=&
	\frac{1}{p(x-1)+1},
	%\label{eq:GammaLog}
\end{eqnarray*} 
as the Gamma function of multiplicative generator, where $x \neq 1 -\frac{1}{p}$ for $p \neq 0$ and $\Gamma_{\pow} \equiv 1$ for $p=0$. 
\section{Pairability}
In this paper we mainly deal with the functional equation
 \begin{equation}
 f(x+y)=f(x) g(y) + f(y) g(x) 
 \tag{S}
\label{eq:SinFG}
 \end{equation}
coming from the sine addition formula.
It is known that the pair of real trigonometric functions $(f,g)=(\sin, \cos)$ satisfy equation \eqref{eq:SinFG} on $\R$.
Moreover, it holds $\cos{x}=\sin{\left(x+\frac{\pi}{2}\right)}$ for all $x\in \R$.
This motivates us to deal with the question whether, for a given function $f$, there is some (possibly non-constant) period $T$  such that
$g(x)=f(x+T)$ for all suitable $x$, and the pair $(f,g)$ satisfies \eqref{eq:SinFG}. 
To make things more precise, we start with the following

\begin{definition}
	Let $I \subset \R$ be an interval closed under addition and $f, g:I \to \R$ be functions.
	The function $f$ is called \text{pairable} with $g$ with respect to the functional equation \eqref{eq:SinFG}, 
	if there is some  $T: I^2 \to \R$, called period function, such that $g(x)=f(x+T)$ for all $x\in I$,
	and the pair $(f(\cdot),f(\cdot+T))$ satisfies \eqref{eq:SinFG}. 
	Moreover, if $f$ satisfies in addition some Cauchy functional equation,
	then $(f,g)$ is called 
	\textit{Cauchy pair} with respect to this functional equation.
	If $(f,g)$ is a Cauchy pair and $g$ satisfies the same Cauchy functional euqation as $f$, 
	then $(f,g)$ is called \textit{true Cauchy pair}.
\end{definition}
The careful reader will have noticed that in the latter definition to some degree the notion of domain is suppressed.
So, when introducing Cauchy pairs,  it is not mentioned on what set the function $f$ satisfies some given Cauchy functional equation.
The reason for this is twofold and we will clarify this later on ad hoc when dealing with concrete examples. 

Obviously, the notion of pairability makes only sense for functional equations with at least two unknown functions,
as also the one 
 \begin{equation}
	g(x+y)=g(x)g(y)-f(x) f(y)  
	\tag{C}
	\label{eq:CosFG}
\end{equation}
coming from the cosine addition formula.

%There are no true regular Cauchy pairs on the whole domain.
%\begin{remark}	
%\end{remark}

%In Section ... the some Cauchy pairs for the cosine addition will be determind.
\section{Preliminaries}
In this paper, functions often occure related to translation or scaling behavior.
It seems natural, especially during the calculations, to suppress sometimes the dependency on the variables.
For instance, the letter $T$ denoting the period function may stand for a number or function of one or two variables,
and to make this more clear, we sometimes write $T=T(x,y)$ or $T=T(x)$ as common in engineering literature or when dealing with parametrized curves.
 With the letter $\R^*$ we denote the set of real numbers without zero.
 
\subsection{Cauchy pairs} 
In this section we prove some properties on Cauchy pairs
and start with the
\subsection{Additive Cauchy equation}
Let $f: \R \to \R$ be a regular solution to the additive Cauchy functional equation
 \begin{equation}
	f(x+y)=f(x) + f(y),
	\tag{A}
	\label{eq:AddiCauchy}
\end{equation}
for all $x,y \in \R$. 
Thus, there is $c \in \R$ such that $f(x)=cx$ for all $x\in \R$.
All Cauchy pairs with respect to the sine addition formula \eqref{eq:SinFG}
are determined in the following
\begin{theorem}
	Let $f:\R \to \R$ be a regular solution to the additive Cauchy functional equation \eqref{eq:AddiCauchy} on $\R$,
	namely there is $c\in \R$  such that $f(x)=cx$ for all $x \in \R$,
	and $g: \R\to\R$ be defined by $g(x)=f(x+T)$ for all $x \in \R$.
	Then $(f,g)$ is a Cauchy pair for the sine addition functional equation \eqref{eq:SinFG} if, and only if,
\begin{equation*}
	T_c(x,y)=\frac{1}{c}- \frac{2xy}{x+y},	
\end{equation*}
for all $x,y \in \R$.
\end{theorem}

\begin{proof}
Take a function $f: \R \to \R$ with $f(x)=cx$ for some $c\in \R$
 and define $g: \R \to \R$ by $g(x):=f(x+T)$ for some $T: \R^2 \to \R$. 
Assume that the pair $(f,g)$ satisfies \eqref{eq:SinFG}.
Hence, for all $x,y \in \R$,
\begin{eqnarray*}
	f(x+y) &=& c(x+y)\\
	& \stackrel{!}{=}& f(x) g(y) + f(y) g(x)\\
	&=&  f(x) f(y+T) + f(y) f(x+T)\\
	&=& cx c(y+T) + cy c(x+T)\\
	&=& c^2 \left( 2xy + T (x+y) \right).
\end{eqnarray*}
For $c=0$ here clearly equality holds. 
Therefore, we assume $c \neq 0$ to devide both sides by it, yielding
\begin{equation}
	x+y = c \left( 2xy + T (x+y) \right).
	\label{eq:tag1}
\end{equation}
Whence,
\begin{equation*}
	\frac{x+y}{c} -2xy =  T (x+y).
\end{equation*}
If $y = -x$, we have $-2xy=0$, which holds iff $x=0$ or $y=0$, so not for all $x,y \in \R$.
Otherwse, i.e., if $x+y \neq 0$,
the period function equals 
\begin{equation*}
	 T=T_c (x,y)= \frac{1}{c} -\frac{2xy}{x+y}, 
\end{equation*}
a non-zero constant minus the harmonic mean, as was claimed. 
The conversion is easy to verify.
\end{proof}
Interestingly, $T_c (x,x)=\frac{1}{c}-x$ and $T_1 (x,x)=1-x$, a term which appears naturally in the reflection formula of the Euler Gamma function.
\\
Note also that from equation \eqref{eq:tag1} we get
\begin{equation*}
	y \left(1-c(2x-T) \right) =  x\left( cT  -1 \right),
\end{equation*}
which can be solved for $y$ only if the bracket on the left hand side does not vanish,
thus if $x \neq \frac{1}{2} \left( \frac{1}{c}+T\right)$.
(Since $T$ depends on both $x$ and $y$, this excludes more than one number!)
Assume that $T$ is independent of, let's say, $x$ .
Then, since $T$ is differentiable, its partial derivative with respect to this variable vanishes
\begin{equation*}
	\frac{\partial T}{\partial{x}} (x,y)=\frac{-2y^2}{(x+y)^2}=0,
\end{equation*}
hence $y=0$ (observe that $T(x,0)=\frac{1}{c}$). 
Since $T$ is symmetric in both variables, $T$ is indipendent of $y$ only for $x=0$.
Thus, $T$ is constant only at the origin $x=y=0$.
En passant we thus proved 
that for regular solutions to the addtive Cauchy equation there is no Cauchy pair with respect to \eqref{eq:SinFG} having constant 
%(i.e. constant in some neighborhood) 
period. 
Since translations of a line through the origin is not additive any more,
there are no true Cauchy pairs of additive type with respect to \eqref{eq:SinFG}.
% (or any other functional equation).
Note also that the period function determined here vanishes exactly on the level sets of the harmonic mean $H$ and $\lim_{c \to \infty} T_c=-H$. 

What, on the other hand, happens, if the roles of $f$ and $g$ are reversed meaning that $g$ is regular additive, the second function is defined by a certain translation of the first one, thus $f(x):=g(x+T)$ for some period function $T$ such that the sine functional equation holds true?
The answer is given in the following
\begin{remark}
Assume that $(g,f)$ is a dual Cauchy pair of additive type with respect to the sine functional equation,
i.e., $g(x)=cx$ for some $c\in \R$, $f(x):=g(x+T)$.
By \eqref{eq:SinFG}, we have, for all $x,y \in \R$,
\begin{eqnarray*}
f(x+y) &=& g(x+y+T)\\
&=& c	(x+y+T)\\
&=& f(x)g(y)+f(y)g(x)\\
&=& g(x+T)g(y)+g(y+T)g(x)\\
&=& c^2(xy+T (x+y)).
\end{eqnarray*}
Without loss of generality we can assume $c \neq 0$.
Thus, for all $x,y \in \R$,
\begin{equation*}
x+y+T=c(xy+T (x+y)),
\end{equation*}
  and hence
\begin{equation*}
	T=\frac{x+y-cxy}{c(x+y)-1}
\end{equation*}
unless $y=\frac{1}{c}-x$ where the period function is not defined.
It is natural to ask when the two period functions of Cauchy pair and dual Cauchy pair, respectively, are equal,
which in this case amounts to the algebraic equation
 \begin{equation*}
 	(x^2 y + y x^2) c (c-1)-c (x^2 +y^2 +4xy)+x+y=0.
 \end{equation*}
\end{remark}
Next all additive Cauchy pairs with respect to the cosine addition formula \eqref{eq:CosFG}
are determined.
\begin{theorem}
	Let $g:\R \to \R$ be a regular solution to the additive Cauchy functional equation \eqref{eq:AddiCauchy} on $\R$,
	namely there is $c\in \R$  such that $g(x)=cx$ for all $x \in \R$,
	and $f: \R\to\R$ be defined by $f(x)=g(x+T)$ for all $x \in \R$.
	Then $(g,f)$ is a Cauchy pair with respect to the cosine addition functional equation \eqref{eq:CosFG} if, and only if,
	\begin{equation*}
		T(x,y)=-\frac{x+y}{2} - \frac{1}{2} \sqrt{(x+y) \left[ (x+y)-\frac{4}{c}\right]},	
	\end{equation*}
	for all $x,y \in \R$.
\end{theorem}

\begin{proof}
	Take a function $g: \R \to \R$ with $g(x)=cx$ for some $c\in \R$
	and define $f: \R \to \R$ by $f(x):=g(x+T)$ for some $T: \R^2 \to \R$. 
	Assume that the pair $(g,f)$ satisfies \eqref{eq:CosFG}.
	Hence, for all $x,y \in \R$,
	\begin{eqnarray*}
		g(x+y) &=& c(x+y)\\
		& \stackrel{!}{=}& g(x) g(y) - f(x) f(y) \\
		&=&  g(x) g(y) - g(x+T) g(y+T)\\
		&=& cx cy - c (x+T) c(y+T)\\
		&=& -c^2 \left( T^2  + T (x+y) \right).
	\end{eqnarray*}
	For $c=0$ clearly equality holds. 
	Therefore, let us assume that $c \neq 0$. Hence, 
	\begin{equation*}
		x+y = -c \left( T^2 + T (x+y) \right),
		%\label{eq:tag1}
	\end{equation*}
	and thus
	\begin{equation*}
		T^2 + (x+y) T +\frac{x+y}{c}=0,
	\end{equation*}
   a quadratic equation in $T$ with solution
   	\begin{eqnarray*}
		T(x,y) &=& -\frac{x+y}{2} \pm \sqrt{\left(\frac{x+y}{2} \right)^2-\frac{x+y}{c}}\\
		&=& -\frac{x+y}{2} \pm \frac{1}{2} \sqrt{(x+y) \left[ (x+y) -\frac{4}{c}\right]}\\
			&=& -\frac{x+y}{2} \pm \frac{1}{2} \sqrt{x \left(x-\frac{4}{c}\right) +2xy +y \left(y-\frac{4}{c}\right)}.
	\end{eqnarray*}
	The conversion is easy to verify.
	\end{proof}
Note that this period function vanishes iff $y=-x$, which obviously, unlike to the case before, is not the level set of a mean. 
Moreover, the term under the square root is reflexive only at the two fixed points $x=0$ and $x=\frac{8}{3}c$.

The special case where the sum of the variables is constant, say $x+y=d$, gives us the period
\begin{equation*}
	T_c(d):=T(x,d-x)=\frac{1}{2}\left(-d \pm \sqrt{d^2-4\frac{d}{c}}\right).
\end{equation*}
In the case $c=d$ we have $T(c):=T(x,c-x)=\frac{1}{2}\left(-c \pm \sqrt{c^2-4}\right)$.
When the period function $c \mapsto T(c)$ attains axtremas values, is of special interest.
Its derivative 
\begin{equation*}
	T^\prime (c)=-\frac{1}{2} \pm \frac{c}{\sqrt{c^2-4}}
\end{equation*}
vanishes exactly for $c^2=-\frac{4}{3}$. The corresponding value of the period function is
\begin{eqnarray*}
	T_E&:=&T \left(\pm \frac{2i}{\sqrt{3}}\right)\\
	&=&\mp \frac{1}{\sqrt{3}}i \pm \frac{2}{\sqrt{3}}i,
\end{eqnarray*}
so $T_E=\pm \frac{\sqrt{3}}{3}i$.

\iffalse
\subsection{Dual Cauchy pairs}
In the context of pairable functions with respect to the sine or cosine functional equation the roles of the functions $f$ and $g$ are not completely symmetric.
This is obvious in the case of equation \eqref{eq:CosFG}, but also on subtstituting $f$ by $g$ and $g$ by $f$, the right hand side of  \eqref{eq:SinFG} remains invariant (since multiplication of real functions is commutative), but the left hand changes.

\fi

\subsection{Exponential Cauchy equation}
Let $f: \R \to (0, + \infty)$ be a regular solution to the exponential Cauchy functional equation
\begin{equation}
	f(x+y)=f(x)  f(y),
	\tag{E}
	\label{eq:ExpoCauchy}
\end{equation}
for all $x,y \in \R$. 
Thus, there is $a \in (0, + \infty)$ such that $f(x)=a^x$ for all $x\in \R$.
All Cauchy pairs with respect to the sine addition formula \eqref{eq:SinFG}
are determined in the following

\begin{theorem}
	Let $f:\R \to (0, +\infty)$ be a regular solution to the exponential Cauchy functional equation \eqref{eq:ExpoCauchy} on $\R$,
namely there is $a \in (0, +\infty)$  such that $f(x)=a^x$ for all $x \in \R$,
and $g: \R\to (0, +\infty)$ be defined by $g(x)=f(x+T)$ for all $x \in \R$.
Then $(f,g)$ is a Cauchy pair for the sine addition functional equation \eqref{eq:SinFG} if, and only if,
\begin{equation*}
	T(x,y)=-\log_a {2},	
\end{equation*}
for all $x,y \in \R$.
Moreover, there are no dual exponential Cauchy pairs of the form $(g,f)$ for the sine functional equation.
\end{theorem}

\begin{proof}
Take a function $f: \R \to (0, +\infty)$ with $f(x)=a^x$ for some fixed $a \in (0, +\infty)$
	and define $g: \R \to \R$ by $g(x):=f(x+T)$ for some $T: \R^2 \to \R$. 
	Assume that the pair $(f,g)$ satisfies \eqref{eq:SinFG}.
	Hence, for all $x,y \in \R$,
	\begin{eqnarray*}
		f(x+y) &=& a^{x+y}\\
		& \stackrel{!}{=}& f(x) g(y) + f(y) g(x)\\
		&=&  f(x) f(y+T) + f(y) f(x+T)\\
		&=& a^x a^{y+T} + a^y a^{x+T}\\
		&=& 2 a^{x+y+T},
	\end{eqnarray*}
thus $2 a^T=1$ and consequently $T=-\log_a {2}$.
To prove the second part, assume that $g(x)=a^x$ is a regular exponential and $f$ is defined by $f(x)=g(x+T)$ for all $x \in \R$ with some $T: \R^2 \to \R$.
By \eqref{eq:SinFG}, this means, for all $x,y \in \R$,
	\begin{eqnarray*}
	f(x+y) &=& a^{x+y+T}\\
	& \stackrel{!}{=}& f(x) g(y) + f(y) g(x)\\
	&=&  g(x+T) g(y) + g(y+T) g(x)\\
	&=& a^{x+T} a^{y} + a^{y+T} a^{x}\\
	&=& 2 a^{x+y+T},
\end{eqnarray*}
whence, $a^T=0$, which is impossible since non-trivial exponentials have no zeros and $T$ is supposed to be real-valued.
\end{proof}
So we observed that there is exactly one constant period function for regular exponential Cauchy pairs,
but sadly no dual Cauchy pair of this type.

Next all regular exponential Cauchy pairs with respect to the cosine addition formula \eqref{eq:CosFG}
are dealt with.
Since  $g$ plays the role of the cosine function and $f$ the one of sine in \eqref{eq:CosFG}, 
we thought it would be natural here to ask for Cauchy pairs in the form $(g,f)$ insteadt of $(f,g)$ for the cosine euquation.
Interestingly, there does not exists any Cauchy pair of this form.
That is why we also deal with Cauchy pairs in the usual form $(f,g)$,
 which we call dual Cauchy pair with respect to \eqref{eq:CosFG} since $g$ is the main actor in the cosine addtion law. 
Here surprisingly two constant period functions exist in terms of the golden ratio.

\begin{theorem}
	Let $g:\R \to (0, +\infty)$ be a regular solution to the exponential Cauchy functional equation \eqref{eq:ExpoCauchy} on $\R$,
	namely there is $a \in (0, +\infty)$  such that $g(x)=a^x$ for all $x \in \R$,
	and $f: \R\to (0, +\infty)$ be defined by $f(x)=g(x+T)$ for all $x \in \R$.
	Then there is no period such that $(g,f)$ is a Cauchy pair for the cosine addition functional equation \eqref{eq:CosFG}.
	Moreover, if the roles of $f$ and $g$ are exchanged, namlely if
	$f(x)=a^x$ and $g(x)=a^{x+T}$, then $(f,g)$ is a Cauchy pair for the cosine addition functional equation
	 if, and only if,
	\begin{equation*}
		T_{1,2}=\log_a \left({\frac{1}{2} \left(1  \pm \sqrt{5} \right)}\right).
	\end{equation*}
\end{theorem}

\begin{proof}
	To proof the second part, take a function $f: \R \to (0, +\infty)$ with $f(x)=a^x$ for some $a \in (0, +\infty)$
	and define $g: \R \to (0, +\infty)$ by $g(x):=f(x+T)$ for some $T: \R^2 \to \R$.
	Assume that the pair $(f,g)$ satisfies \eqref{eq:CosFG}.
	Hence, for all $x,y \in \R$,
	\begin{eqnarray*}
		g(x+y) &=& a^{x+y+T}\\
		& \stackrel{!}{=}& g(x) g(y) - f(x) f(y)\\
		&=&  a^{x+T} a^{y+T} -a^x a^y\\
		&=& a^{x+y+2T} - a^{x+y}\\
		&=&  a^{x+y} \left(a^{2T}-1\right),
	\end{eqnarray*}
	thus $a^T=a^{2T}-1$. 
	Putting $q:=a^T$ we get $q^2-q-1=0$ implying that $q_{1,2}=\frac{1}{2} \left( 1  \pm \sqrt{5}\right)$ is the golden ratio.
	Consequently, $T=-\log_a{2} + \log_a{\left( 1  + \sqrt{5}\right)}$.
	(Allowing also complex-valued period functions, we would also get $T=-\log_a{2} + \log_a{\left( 1  - \sqrt{5}\right)}$).
	Vice versa, namely when looking for exponential Cauchy pairs of the form $(g,f)$ with $g(x)=a^x$ for some $a \in (0, +\infty)$ and $f(x):=g(x+T)$ for
	the cosine equation, from \eqref{eq:CosFG}, we obtain for all $x,y \in \R$,
	\begin{eqnarray*}
		g(x+y) &=& a^{x+y}\\
		& \stackrel{!}{=}& g(x) g(y) - f(x) f(y)\\
		&=&  a^{x} a^{y} -a^{x+T} a^{y+T}\\
		&=& a^{x+y} - a^{x+y+2T}\\
		&=&  a^{x+y} \left(1 - a^{2T} \right),
	\end{eqnarray*}
	implying $1=1 - a^{2T}$ and thus $-a^{2T}=0$, which has no number as solution exept possibly $T=\log_a{0}=-\infty$.
\end{proof}

\subsection{Multiplicative Cauchy equation}
Let $f: (0, +\infty) \to (0, +\infty)$ be defined $f(x)=x^p$ for some $p \in \R$ fixed.
When looking for Cauchy pairs of form $(f,g)$, we determine 
all period functions $T:(0, +\infty)^2 \to (0, +\infty)$ such that $g: (0, +\infty) \to (0, +\infty)$ is defined by $g(x):=f(x+T)$
and the pair $(f,g)$ satisfies \eqref{eq:SinFG}, thus, for all $x,y \in (0, +\infty)$,

\begin{eqnarray*}
	f(x+y) &=& (x+y)^p\\
&=& f(x) g(y) + f(y) g(x)\\
&=& f(x) f(y+T) + f(y) f(x+T)\\
&=&	x^p (y+T)^p + y^p (x+T)^p\\
&=&	\left[x (y+T)\right]^p + \left[y (x+T)\right]^p.
\end{eqnarray*}
To solve here for $T$ is not always possible.
% (explicitly or even theoretically, as an application of the Inverse Function Theorem shows).
The case $p=1$ appears also in the additive Cauchy equation.
Let us consider $p=2$.
An easy calculation shows that
\begin{align*}
	T(x,y)=\frac{1}{x^2 + y^2} \left[-(x^2y + y x^2) \pm \sqrt{2(xy)^3 + (x+y)^2 (x^2 +y^2)} \right].
\end{align*}
The case $p=3$ is already in\textit{}volved.
It turns out that it is much easier to look for a scaled version of the given function, 
i.e., for given $f(x)=x^p$ find all \textit{scaleabiliy functions} $t: (0, +\infty)^2 \to (0, +\infty)$ such that $g: (0, +\infty)^2 \to (0, +\infty)$ defined by $g(x):=f(tx)$
satisfies \eqref{eq:SinFG}, thus, for all $x,y \in (0, +\infty)$,
\begin{eqnarray*}
	f(x+y) &=& (x+y)^p\\
	&=& f(x) g(y) + f(y) g(x)\\
	&=& f(x) f(ty) + f(y) f(tx)\\
	&=&	x^p (ty)^p + y^p (tx)^p\\
	&=&	2\left[ txy \right]^p.
\end{eqnarray*}
Taking here the $p$-th root on both sides, we obtain $x+y=2^{\frac{1}{p}} xyt$ and thus
\begin{align*}
	t=\frac{x+y}{2^{\frac{1}{p}} xy}.
\end{align*}
Under suitable assumptions on domain and range of the involved functions, scaled and translated versions of a given function are related,
and it is easy to prove the following

\begin{theorem}
	Let $I \subset \R$ be an interval closed under addition, $f, g: I \to (0, +\infty)$ be an (S)- or a $(C)$-pair.
	Then $T: I^2 \to \R$ is a period function for $(f,g)$ 
	iff so is the scaleabiliy function $t:=\exp \circ T$ for  $(f \circ \log, g \circ \log)$.
\end{theorem}
In a nutshell, real translativity is related to positive scaleabiliy, but these problems are not equivalent 
since also the involved functions are not the same, which sadly means also that the knowledge of the scaleability function
for power functions above does not give us automatically the corresponding period function.

\subsection{Sine and Cosine Representer}
The sine and cosine additivity laws involve some degree of freedom in the sense that for a given function, one may calculate the other one.
To make things precise, we start with the following

\begin{definition}
Let $I \subset \R$ be an interval closed under addition and $f: I\to \R$. A function $g: I \to \R$ such that the sine addition law \eqref{eq:SinFG} [the cosine addition law \eqref{eq:CosFG}] holds true is called sine respresentor [cosine representor] of the function $f$.
\end{definition}

Under some conditions, sine and cosine represenator are uniquely determined, which is formulated precisely in the followig two remarks.
\begin{remark}["Away from the zeros the sine representer is unique."] 
Let $f:I \to (0, +\infty)$ be a function and $g: I \to \R$ a sine representaor of $f$, i.e.
\begin{align*}
	f(x+y)=f(x) g(y) + f(y) g(x), \qquad x,y \in I.
\end{align*}
Setting $x=y$, it follows that $f(2x)=2 f(x) g(x)$.
Consequently, the sine representor of $f$ is given by
\begin{align*}
	g(x)=\frac{f(2x)}{2f(x)}, \qquad x \in I.
\end{align*}
Thus, we may write the sine addition functional equations as 
\begin{align*}
	f(x+y)=f(x) \frac{f(2y)}{2f(y)} + f(y) \frac{f(2x)}{2f(x)}, \qquad x,y \in I.
\end{align*}
\label{rem:SinRep}
\end{remark}

\begin{remark} ["The cosine representer is unique."]
	\label{rem:CosRep}
	Let $g:I \to \R$ be a function and $f: I \to \R$ a cosine representer of $g$, i.e.
\begin{align*}
	g(x+y)=g(x) g(y) - f(x) f(y), \qquad x,y \in I.
\end{align*}
Setting $x=y$, it follows $g(2x)=\left(g(x)\right)^2  -  \left(f(x)\right)^2$.
	Consequently, the cosine representer of $g$ is given by
	\begin{align*}
		f(x)=\pm \sqrt{\left(g(x)\right)^2 - g(2x)}, \qquad x \in I.
	\end{align*}
(The sign of the term under the square root is of special interest!)
We may write the cosine addition functional equations as 
\begin{align*}
	g(x+y)=g(x) g(y) - \sqrt{\left(g(x)\right)^2 - g(2x)} \sqrt{\left(g(y)\right)^2 - g(2y)},
\end{align*}
for all $x,y \in I$ such that $\left(\left(g(x)\right)^2 - g(2x)\right) \left( \left(g(y)\right)^2 - g(2y)\right) \geq 0$.
\end{remark}
Note that the moral from the square bracket of the latter remark should rather be:
"The complex-valued cosine representer is unique."
If the cosine representer $f:=g_C$ is assumed to be real-valued,
we need the condition ${\left(g(x)\right)^2 - g(2x)} \geq 0$ to hold on the interval $I$.

Since the quotient of two even functions or two odd functions is odd and sum and difference of even functions is even, we get the following
\begin{remark}
	Let $f:I \to (0,+\infty)$ be an even or an odd function defined on a real interval symmetric with respect to $0$.
	Then its sine represenator $f_S$ is even.
	Moreover, if $f$ is even, then so is its cosine representator $f_C$ .
\end{remark}

\subsection{Sine and Cosine Representer of smooth Cauchy functions}
To get a feeling for sine and cosine representer, we calcute for the regular solutions to one of the Cauchy equations their sine and cosine representer.
The sine representer $f_S$ is given by the formula $f_S (x)=\frac{f(2x)}{2 f(x)}$
and its cosine representer $f_C$ by $f_C (x)=\pm \sqrt{\left(f(x)\right)^2 -f(2x)}$ (see Remark \ref{rem:SinRep} and \ref{rem:CosRep}).

\begin{remark}
	Let $f: \R \to \R$ be a non-trivial regular additive function, which means that there is $c \in \R^*$, such that $f(x)=cx$ for all $x\in\R$.
	Then $f_S (x)= 1$ for $x\in \R^*$ and 
	$f_C (x)= \pm \sqrt {cx (cx-2)}$ for $x \in (-\infty, \min(0, \frac{2}{c})) \cup ( \max(0, \frac{2}{c}), +\infty )$.
	The value at zero of the sine representer $f_S$ is undertermined; by the de L'Hospital rule it is reasonable to put $f_S (0)=1$.
	Also we need to discuss the situation when the radicand is negative for the cosine representer.
\end{remark}

\begin{remark}
	Let $f: \R \to (0,+\infty)$ be a non-trivial regular exponential function, which means that there is a positive real number $a$ such that $f(x)=a^x$ for all $x\in\R$.
	Then $f_S (x)= \frac{1}{2} f(x)$ for $x \in \R$ and 
	$f_C (x)= 0$ for $x \in \R$.
\end{remark}

\begin{remark}
	Let $f: (0,+\infty) \to \R$ be a non-trivial regular logarithmic function; thus there is $c \in \R$ such that $f(x)=c\log{x}$ for all $x\in (0, +\infty)$.
	Then $f_S (x)= \frac{1}{2} \left( \frac{\log{2}}{\log{x}}+1\right)$ for $x \in (0, +\infty)$, $x \neq 1$; 
	moreover, $f_C(1)=\frac{\log{(y+1)}}{\log{y}}$ for $y \in (0, +\infty)$, $y \neq 1$.
	 The cosine representer of $f$ reads
	$f_C (x)= \pm \sqrt{c\log{x} (c\log{x}-1)-c\log{2}}$ for $x \in (0, +\infty)$. 
	If $f_C$ is supposed to be real-valued, more restrictions on the domain are needed.
\end{remark}

\begin{remark}
	Let $f: (0,+\infty) \to (0,+\infty)$ be a non-trivial regular multiplicative function; thus there is $p \in \R$ such that $f(x)={x}^p$ for all $x\in (0, +\infty)$.
	Then $f_S (x)= 2^{p-1}$ for $x \in (0, +\infty)$ and 
	$f_C (x)= 0$ for $x \in (0, +\infty)$. 
\end{remark}

\subsection{Elementary properties of pairability}

It is obvious that pairability with identically vanishing period function $T$ implies $f=g$.

By \eqref{eq:SinFG}, this yields
\begin{equation*}
	f(x+y)=2 f(x) f(y),
\end{equation*}
which means that the function $2f$ is exponential.
Hence, if $f$ and $g$ are trivially pairable (i.e. $T \equiv 0$) with respect \eqref{eq:SinFG}, then $2f$ ($= 2g$) satisfies the Cauchy exponential equation.

Secondly, trivial pairability of $f$ and $g$ with respect to the cosine addition formula \eqref{eq:CosFG} yields
\begin{equation*}
	g(x+y)=0,
\end{equation*}
which means that the function $g$ vanishes identically%
\footnote{Sometimes it is common to introduce some parameter into the cosine addition formula (see for instance \cite{Stetkaer} ) and we deal with $g(x+y)=g(x)g(y)+c f(x)(y))$. In this case trivial pairability yields $f(x+y)=(1+c) f(x) f(y)$ meaning that the function $(1+c) f$ is exponential.}.
On the hand, if the period function of an (possibly non-trivial) $(S)$-pair [$(C)$-pair] has a zero at $(x,y)$,
 then $2f$ is (locally) exponential [$g$ has a zero at $x+y$].

We can formulate pairability with respect to the sine or cosine addition formulas in terms of elementary iterative functional equations.

\begin{remark}
Let $(f,g): I^2 \to \R$	be pairable with respect to the sine addition functional equation \eqref{eq:SinFG}.
By definition, this means that there exists some period function $T: I^2 \to \R$ such that, for all $x,y \in I$,
\begin{eqnarray*}
	f(x+y) &=& f(x) g(y) + f(y) g(x)\\
	&=& f(x) f(y+T) + f(y) f(x+T).
\end{eqnarray*}
Setting here $x=y$, we obtain $f(2x)=2f(x) f(x+T)$ and, on replacing $x$ by $\frac{x}{2}$, thus $f(x)=2 f(\frac{x}{2}) f(\frac{x}{2}+T)$.
Assuming that $f$ has no zeros, we may also conclude $f(x)=\frac{f(2x)}{2f(x+T)}$.
Moreover, we obtain under suitable assumptions (invertibility and, for instance, positivity of $f$) explicitly the period function 
\begin{equation*}
T(x)=f^{-1} (f_S(x))-x.
\end{equation*}
\end{remark}

\begin{remark}
	Let $(g,f): I^2 \to \R$	be pairable with respect to the cosine addition functional equation \eqref{eq:CosFG}.
	By definition, this means that there exists some period function $T: I^2 \to \R$ such that, for all $x,y \in I$,
	\begin{eqnarray*}
		g(x+y) &=& g(x) g(y) - f(x)f(y) \\
		&=& g(x) g(y) - g(x+T) g(y+T).
	\end{eqnarray*}
	Setting here $x=y$, we thus obtain $g(2x)=\left(g(x)\right)^2 - \left(g(x+T)\right)^2$ and, on replacing $x$ by $\frac{x}{2}$, 
	we get $g(x)=\left(g(\frac{x}{2})\right)^2 - \left(g(\frac{x}{2}+T)\right)^2$.
	We may also conclude $g(x)= \pm \sqrt{g(2x)+\left(g(x+T)\right)^2}$
	and, under additional assumptions,
	\begin{equation*}
	T(x)=g^{-1} \left(g_C(x) \right) -x,
	\end{equation*}
where $g_C(x):=\pm \sqrt{(g(x))^2-g(2x)}$ is the cosine representer of $g$.
\end{remark}

Obviously the formulas for the periodicity function presented in the last two remarks 
work only in cases where the period function does not depend on both variables
(compare this to the theorems where the period function was effectively calculated
 for the regular Cauchy pairs).
 
 That pairable functions are periodic in a generalized sense (so with possibly non-constant period) is proved in the following
 \begin{remark}
 Let $(f,g)$ be an $(S)$-pair defined on an interval containing zero with $f(0) \neq 0$, i.e., $g(x)=f(x+T)$ for some period function $T$ and
 \begin{equation*}
 	f(x+y)=f(x) f(y+T)+f(y) f(x+T), \qquad \text{ for all } x,y.
 	\end{equation*}
 Setting here $y=0$, we obtain
\begin{equation*}
	f(x)=f(x) f(T)+f(0) f(x+T), \qquad \text{ for all } x.
\end{equation*}
 and thus, 
 \begin{equation*}
 f(x+T)=	f(x) \frac{1-f(T)}{f(0)}, \qquad \text{ for all } x.
 \end{equation*}
With $c:=\frac{1-f(T)}{f(0)}$ we have
\begin{equation*}
	f(x+T)=	f(x) c, \qquad \text{ for all } x.
\end{equation*}

 Hence, up to the constant $c$ (which is not necessarily a constant), the function $f$ is $T$-periodic.
 In the case $c=1$, which amounts to the condition $f(0)+f(T)=1$, the function $f$ is $T$-periodic in the usual sense.
 
 Note that in general the period $T$ is not assumed to be constant,
 which means that the here stated periodicity ought to be understood in a broader sense.
 
 The same holds for a $(C)$-pair:
 % (with reciprocal constant):
 Let $(g,f)$ be an $(C)$-pair defined on an interval containing zero, i.e., $f(x)=g(x+T)$ for some period function $T$ and
 \begin{equation*}
 	g(x+y)=g(x) g(y)+g(x+T) g(y+T), \qquad \text{ for all } x,y.
 \end{equation*}
 Setting $y=0$, we obtain
 \begin{equation*}
 	g(x)=g(x) g(0)+g(x+T) g(T), \qquad \text{ for all } x,
 \end{equation*}
 and thus
 \begin{equation*}
 	g(x+T)=	g(x) \frac{1-g(0)}{g(T)}, \qquad \text{ for all } x.
 \end{equation*}

 Hence, up to a constant, $g$ is $T$-periodic.
 If $g(0)+g(T)=1$, we have usual periodicity; otherwise it should be understood in a generalized sense.
 
 \end{remark}

\subsection{Pairability as an  Equivalence Relation}
It may seem very natural to ask whether pairability of functions has the usual properties of an equivalence relation.
%In general, this is not the case.
We start with symmetry, and thus ask for a characterization of the situation when a function $f$ is pairable with $g$ implies that
 $g$ is pairable with $f$. 
 In a nutshell, only for trivially pairable functions the notion of pairability is symmetric.
 We prove this pars pro toto for the sine addition formula in the following

\begin{remark}
Let $f$ be pairable with $g$ with respect to the sine addition functional equation.
 Thus, there is a period function $T$ such that
$g(x)=f(x+T)$ for all suitable $x$, and, on substituting here $x$ by $x-T$, we get $f(x)=g(x-T)$. 
By  \eqref{eq:SinFG}, we have, for all $x,y$ from the domain of $f$,
\begin{equation*}
f(x+y)= f(x) f(y+T)+ f(y) f(x+T).	
\end{equation*}
Hence, 
\begin{equation*}
	g(x+y-T)= g(x-T) g(y)+ g(y-T) g(x).	
\end{equation*}
If we assume on the other hand that $g$ is pairable with $f$ (with respect to (S), of couse), we get that there is some $\bar{T}$ such that
\begin{equation*}
	g(x+y)=  g(x) g(y+\bar{T}) +g(y)g(x+\bar{T}).	
\end{equation*}
The right hand sides here have the same form; setting $T=-\bar{T}$, they coincide. 
But at the left hand side of the previous equation there is a problem unless ${T}=0$.
Thus, $(f,g)$ is an $(S)$-pair implies that $(g,f)$ is an $(S)$-pair only in the case $f=g$, 
 thus in the case of trivial pairability. 
By \eqref{eq:SinFG}, this means that the function $2f$ is exponential.
\end{remark}
The analoguous result for the cosine functional equation holds true.

\section{Periodicity of sine representers}
Away from the zeros of $f$, the sine representer is given by
$f_S (x):=g(x)= \frac{f(2x)}{2 f(x)}$.
On the other hand, when dealing with pairable functions, the second function $g$ is defined in terms of the first by $g(x):=f(x+T)$
and the question when sine representers of pairable functions are $T$-periodic amounts to the equation
\begin{equation}
	\frac{f(2x)}{2 f(x)} = f(x+2T).
	\label{eq:periodicitySineRep}
\end{equation}
(Alternatively we could also  work with the 'balanced version' of the equation given by $\frac{f(2(x+T))}{2f(x+T)}=f(x+T)$.)
Analogously, for the cosine representer we have  $g_C(x):=f(x)= \pm \sqrt{(g(x))^2 -g(2x) }$ (cf. Remark \ref{rem:SinRep} and \ref{rem:CosRep}),
and on the other hand, when assuming that $g$ is pairable with $f$, 
which means that there exists some (possibly non-constant) $T$ such that $f(x)=g(x+T)$ for all suitable $x$,
we may thus express the $T$-periodicity of $f$ by
\begin{equation}
	\pm \sqrt{(g(x))^2 -g(2x) } = g(x+2T).
	\label{eq:periodicityCosineRep}
\end{equation}
(Just as above, we could also use the equation $\pm \sqrt{(g(x+T))^2 -g(2(x+T)) } = g(x+T)$ for all suitable $x$ 
to express the $T$-periodicity of the cosine representer $g_C=f$, but we prefer not to, 
since this  would introduce even more difficulties regarding domains.)

Our next remark deals with the periods of the sine representers in connection with regular non-trivial%
\footnote{Constant functions satisfy all four Cauchy equations. Thus, we assume here without loss of generality that all occuring constants are real and non-zero. This should also be clear since these numbers often appear in the denominators.} 
Cauchy pairs.
\begin{remark}
	Let $c, p, a \in \R \setminus \{ 0 \}$, $a>0$, be fixed.
	\begin{enumerate}
		\item
	The power function $f(x)=x^p$ on $(0,+\infty)$ is pairable  with its sine representer $g(x)=2^{p-1}$, 
	 if $T=\frac{1}{\sqrt[p]{2}} -\frac{x}{2}$.
	\item
	The regular additive function $f(x)=cx$ on $\R \setminus \{0\}$ is pairable with its sine representer $g(x)=1$, 
	if $T=\frac{1}{2} \left(\frac{1}{c} - x\right)$.
	
	\item
	The regular logarithmic function $f(x)=c \log{x}$ on $(0, +\infty)$ is pairable with its sine representer 
  $g(x)=\frac{1}{2} \left( \frac{\log{2}}{\log{x}}+1  \right)$, 
	if $T=\frac{1}{2} \left(\e^{\frac{1}{2c} \left(\frac{\log{2}}{\log{x}}+1 \right)}-x \right)$.
	
	\item
	The regular exponential function $f(x)=a^{x}$ on $\R$ is pairable with its sine representer 
	$g(x)=\frac{1}{2} a^{x}$, 
	if $T=-\frac{1}{2} \log_a{2}$.

	\end{enumerate}
\end{remark}

Similarly, we have for the cosine representers the following
\begin{remark}
	Let $c, p, a \in \R \setminus \{0\}$, $a>0$, be fixed.
	\begin{enumerate}
		\item
		The power function $g(x)=x^p$ on $(0,+\infty)$ is pairable with its cosine representer $f(x)=0$, 
		if $T=-\frac{1}{2} x$.
		\item
		The regular additive function $g(x)=cx$ on $\R \setminus \{0\}$ is pairable with its cosine representer $f(x)=\pm \sqrt{(cx)^2- c 2x}$, 
		if $T=\frac{1}{2} \left(-x \pm   | {x} |  \sqrt{1-\frac{1}{c}} \right)$.
		\item
		The regular logarithmic function $g(x)=c \log{x}$ on $(0, +\infty)$ is pairable with its cosine representer 
		$f(x)=\pm \sqrt{(c\log{x})^2-c\log(2x)}$, 
		if $T=\frac{1}{2} \left( \e^{\pm   \sqrt{(\log{x})^2- \frac{1}{c}\log{ (2x)}}}-x \right)$.
		\item
		The regular exponential function $g(x)=a^{x}$ on $\R$ is pairable with its cosine representer 
		$f(x)=0$, 
		if $T=-\infty$.
	\end{enumerate}
\end{remark}

\begin{proof}
	\begin{enumerate}
	\item
Assume that $g$ being the the given power function is pairable with its cosine representer $f$.
By \eqref{eq:periodicityCosineRep} we have
\begin{equation*}
	0=(x+2T)^p,
\end{equation*}
and thus $T=-\frac{1}{2} x$ as claimed.
	\item
Assume that the smooth additive function $g:\R^* \to \R$ defined by $g(x)=cx$ is pairable with the cosine representer $f$.
	By \eqref{eq:periodicityCosineRep} we have
	\begin{equation*}
\pm \sqrt{(cx)^2- c 2x}=c(x+2T).
	\end{equation*}
	Squaring both sides gives us
	\begin{equation*}
		{(cx)^2-  2cx}=c^2(x^2+4xT +4T^2),
	\end{equation*}
and thus
\begin{equation*}
	4c^2 T^2 +4xc^2 T+2cx=0.
\end{equation*}
Since we assumed $c\neq 0$, we may devide here by $4c^2$ to yield
\begin{equation*}
	 T^2 +x T+\frac{x}{2c}=0.
\end{equation*}
This gives us
	$T=\frac{1}{2} \left(-x \pm   | {x} |  \sqrt{1-\frac{1}{c}} \right)$.	
		\item
	Assume that the smooth logarithmic function $g : (0, +\infty) \to \R$ defined by $g(x)=c \log{x}$ is pairable with the sine representer $f$.
	By \eqref{eq:periodicityCosineRep} we have
	\begin{equation*}
		\pm \sqrt{(c\log{x})^2- c\log{ (2x)}}=c \log{(x+2T)},
	\end{equation*}
	and hence $T=\frac{1}{2} \left( \e^{\pm   \sqrt{(\log{x})^2- \frac{1}{c}\log{ (2x)}}}-x \right)$.
	\item
	Assume that the smooth exponential function $g: \R \to (0, +\infty) \to $ defined by $g(x)=a^{x}$ is pairable with the sine representer $f$.
	By \eqref{eq:periodicityCosineRep} we have
	\begin{equation*}
		0= a^{x+2T}.
	\end{equation*}
	Thus $a^{2T}=0$ and $T=-\infty$.
\end{enumerate}
\end{proof}

\section{Conclusion}
The notion of pairability of functions opens a broad field of research
not limited to only functional equations. Of course, also other additivity laws may be the starting point for new and fruitful investigations.
We hope to have shown with this paper 
that exploring the properties of functions having a (possibly non-constant) period induced by some additivity law or other suitable conditions
is indeed of great interest.
Up to now, we can not fully exclude that the devellopment of a certain duality theory based on the notion of pairability with respect to some additivity law
will give deep insights to some of the most difficult unsolved problems in mathematics.

We thank the reader for helpful remarks and comments.

********************************************************************

\end{document}